\documentclass{amsart}
\usepackage{amssymb}
\usepackage{graphicx}
\usepackage[all]{xy}
\usepackage{a4wide}
\usepackage{xcolor}
\usepackage{hyperref}
\newtheorem{lemma}{Lemma}[section]
\newtheorem{theorem}[lemma]{Theorem}
\newtheorem{corollary}[lemma]{Corollary}
\newtheorem{fact}[lemma]{Fact}
\newtheorem{proposition}[lemma]{Proposition}
\theoremstyle{remark}

\newtheorem{remark}[lemma]{Remark}

\newtheorem{examples}[lemma]{Examples}
\theoremstyle{definition}

\newtheorem{definition}[lemma]{Definition}

\renewcommand{\phi}{\varphi}
\newcommand{\Ind}{{\rm Ind}}

\def\normal{\triangleleft}

\def\Agag{{\bar{A}}}
\def\calE{\mathcal{E}}
\newcommand{\bbN}{\mathbb{N}}
\def\mugag{{\bar \mu}}

\def\alphagag{{\bar \alpha}}

\def \bbR {{\mathbb R}}
\newcommand{\calN}{\mathcal{N}}

\newcounter{Lcounter}
\renewcommand{\theLcounter}{\theequation\alph{Lcounter}}

\newenvironment{labellist}{
    \refstepcounter{equation}
    \begin{list}{{\rm(\theLcounter)}}{\usecounter{Lcounter}\leftmargin=5pt}
    }{\end{list}
} %

\begin{document}

\title{Subgroups of profinite surface groups}%

\author{Lior Bary-Soroker}%
\address{
Institut f\"ur Experimentelle Mathematik,
Universit\"at Duisburg-Essen,
Ellernstrasse 29,
D-45326 Essen,
Germany}
\email{lior.bary-soroker@uni-due.de}%

\author[K. F. Stevenson]{Katherine F. Stevenson}
\address{California State University}
\email{katherine.stevenson@csun.edu}

\author{Pavel Zalesskii}%
\address{
Department of Mathematics\\
 University of Bras\'{\i}lia\\
Bras\'{\i}lia-DF 70910-900\\
 Brazil}
\email{pz@mat.unb.br}%

%
%\thanks{}%
%\keywords{}%

\date{}%
%\dedicatory{}%
%\commby{}%
% ----------------------------------------------------------------
\begin{abstract}
We study the subgroup structure of the \'etale fundamental group $\Pi$ of a projective curve over an algebraically closed field of characteristic $0$. We obtain an analog of the diamond theorem for $\Pi$. As a consequence we show that most normal subgroups of infinite index are semi-free. In particular every proper open subgroup of a normal subgroup of infinite index is semi-free. 
\end{abstract}
\maketitle

\section{Introduction}
Every subgroup of a free group is free, this is the content of the Nielsen-Schreier theorem. The profinite version of the Nielsen-Schreier theorem fails in general and even fails for normal subgroups, for example $\mathbb{Z}_p\leq \widehat{\mathbb{Z}}$.  Therefore the question of finding conditions under which a subgroup of a free profinite group is free is natural and of importance. The question was considered by Melnikov, Lubotzky, van der Dries, Jarden, Haran, and others (\cite[Chapter 8]{RibesZalesskii2010} and \cite[Chapter 25]{FriedJarden2008}). 

Roughly speaking the most general criteria are Melnikov's characterization of normal (and accessible) subgroups of free profinite groups and Haran's diamond theorem. In this work we consider the \'etale fundamental group $\Pi = \pi_1(X)$, where $X$ is a curve over an algebraically closed field of characteristic $0$ of genus $\geq 2$.  

If $X$ is affine, then $\Pi$ is free of finite rank. Therefore Melnikov's characterization is known to hold \cite[Chapter 8.6]{RibesZalesskii2010}  and similarly Haran's diamond theorem \cite{Bary-Soroker2006}. If $X$ is projective, then $\Pi$ is a profinite surface group, i.e., the profinite completion of a surface group. Melnikov's characterization for normal subgroups of $\Pi$ is obtained in \cite{Zalesskii2005}. The objective of this work is to obtain the diamond theorem for profinite surface groups:

\begin{theorem}
\label{thm:freesubgroup}
Let $\Pi$ be a profinite surface group of genus $g\geq 2$ and let $N$ be a subgroup of $\Pi$ with $[\Pi:N]=\prod_{p}p^\infty$ as supernatural numbers, where $p$ runs over all primes. Assume there exist normal subgroups $K_1,K_2$ of $\Pi$ such that $K_1\cap K_2\leq N$ but $K_1\not\leq N$ and $K_2\not\leq N$. Then $N$ is a free profinite group of countable rank. 
\end{theorem}

We note that a necessary condition for a profinite group to be free is that it is projective, and a subgroup $N$ of a profinite surface group $\Pi$ is projective if and only if $[\Pi:N] = \prod_{p}p^\infty$ as supernatural numbers, where $p$ runs over all primes \cite[Proposition~1.2]{Zalesskii2005}.

Recently a notion of ``free not necessarily projective" profinite groups evolved from Galois theory \cite{HarbaterStevenson2005, Bary-SorokerHaranHarbater}, the so called semi-free groups. Using this notion we can generalize Theorem~\ref{thm:freesubgroup} to any closed subgroup of infinite index:

\begin{theorem}
\label{thm:semifreesubgroup}
Let $\Pi$ be a profinite surface group of genus $g\geq 2$ and let $N$ be a closed subgroup with $[\Pi:N]=\infty$. Assume there exist normal subgroups $K_1,K_2$ of $\Pi$ such that $K_1\cap K_2\leq N$ but $K_1\not\leq N$ and $K_2\not\leq N$. Then $N$ is semi-free of countable rank. 
\end{theorem}

Since a semi-free projective group is free \cite[Theorem~3.6]{Bary-SorokerHaranHarbater}, Theorem~\ref{thm:freesubgroup} follows from Theorem~\ref{thm:semifreesubgroup}.

A consequence of Theorem~\ref{thm:semifreesubgroup} is that `most' normal subgroups of $\Pi$ of infinite index are semi-free in the following sense. 

\begin{corollary}
\label{cor:LubvdDries}
Let $\Pi$ be a profinite surface group of genus $g\geq 2$ and let $N$ be a closed subgroup with $[\Pi:N]=\infty$. Then every proper open subgroup of $N$ is semi-free.
\end{corollary}

We give more examples in Section~\ref{sec:examples}.

A typical example of a normal subgroup which is not semi-free is the kernel $M$ of the epimorphism from $\Pi$ to its maximal pro-$p$ quotient. Note however that $M$ is contained in a semi-free normal subgroup of $\Pi$. Indeed, there exists an epimorphism $\alpha\colon \Pi\to \mathbb{Z}_p^2$, so $\ker\alpha = K_1 \cap K_2$, where $K_1,K_2$ are normal subgroups of $\Pi$ with $\Pi/K_i\cong\mathbb{Z}_p$. By Theorem~\ref{thm:semifreesubgroup}, $\ker\alpha$ is semi-free, and clearly $M\leq \ker\alpha$. 

We show in fact that every normal subgroup $N$ of $\Pi$ of infinite index such that $\Pi/N$ is not hereditarily just infinite  is contained in a normal semi-free subgroup. (An infinite profinite group is just infinite if it has no proper infinite quotient. It is hereditarily just infinite if every open normal subgroup of it is just infinite.)

\begin{theorem}
\label{thm:sub-semi}
Let $\Pi$ be a profinite surface group of genus $g\geq 2$ and let $M$ be a closed subgroup with $[\Pi:M]=\infty$ such that $\Pi/M$ is not hereditarily just infinite. Then there exists a normal semi-free subgroup $N$ of $\Pi$ such that $M\leq N$. 
\end{theorem}

\section{Surface groups}
The fundamental group $\pi_1(X)$ of an oriented Riemann surface $X$ of genus $g$ is given by the presentation
\[
\pi_1(X) = \Big< x_1,\ldots, x_{g}, y_{1},\ldots,y_{g} \Big| \prod_{i=1}^g [x_i,y_i]\Big>.
\]
Here $[x,y]=x^{-1}y^{-1}xy$. A group with this presentation is  said to be a surface group of genus $g$. We shall call its profinite completion $\Pi$ a profinite surface group of genus $g$. 

\begin{fact}\label{fact:gen-grow}
Let $\Pi$ be a profinite surface group of genus $g$ and let $U$ be an open subgroup of index $n$. Then $U$ is a surface group of genus $n (g-1) +1$.    
\end{fact}

This is well know for surface groups, hence follows for profinite surface groups by completion.

Let $\Pi$ be a profinite group. 
A finite split embedding problem (\textbf{FSEP}) for $\Pi$ consists of finite groups $A,G$, an action of $G$ on $A$, and epimorphisms $\mu\colon \Pi\to G$ and $\alpha\colon A\rtimes G \to G$. We denote it by $(\mu,\alpha)$. A \textbf{weak solution} is a homomorphism $\psi\colon \Pi\to A\rtimes G$ such that $\alpha\circ\psi = \mu$.  If $\psi$ is surjective we say it is a \textbf{proper solution}.

We shall need the following technical lemma. 
\begin{lemma}
\label{lem:large-g-sol}
Let 
$(f\colon \Pi \to B, \alpha\colon A\to B)$ be a  finite split embedding problem for $\Pi$ of genus $g\geq 2 |A|^3$. Then $(f,\alpha)$ is properly solvable. 
\end{lemma}

\begin{remark}
The bound $g \geq 2|A|^3$ is not the best possible. In fact, if $s$ is the minimal number of generators of $\ker \alpha$ as a normal subgroup of $A$, then $g \geq s |B|^2 (|A| + 1)$ suffices. We will not use this sharper bound here, and hence will not prove it.   
\end{remark}

\begin{proof}
Let  $n=|A|$, and $\beta\colon B\to A$ a section of $\alpha$. Note that $\ker \alpha$ is generated by $\frac{|A|}{|B|}$ elements. 
Let $\phi = \beta \circ f \colon \Pi \to A$. Then $\phi$ is a weak solution. 

By \cite[Lemma~6.1]{	PachecoStevensonZalesskii2009}, it suffices to replace the generators of $\Pi$ with a different set of generators having the same unique relation such that
the first $\frac{|A|^2 + |A|}{|B|}\leq \frac{2|A|^2}{|B|}$ new $x_i$'s (resp., $y_i$'s) have the same image under $\phi$. Let $r = \frac{2|A|^2}{|B|}$. 

Each  of the $g$ pairs $(x_i,y_i)$ has $|B|^2$ possibilities for $(\phi(x_i), \phi(y_i))$, hence, since $g\geq 2|A|^3 \geq |B|^2 r$, Dirichlet's box principle gives  indexes $j_{1} < \cdots < j_{r}$ for which 
\begin{equation}
\label{eq:equal-image}
\phi(x_{j_1}) = \cdots = \phi(x_{j_r}) \qquad \mbox{and} \qquad \phi(y_{j_1}) = \cdots = \phi(y_{j_r})
\end{equation}
The following argument explains how to replace $j_1$ with $1$, $j_2$ with $2$, and so forth. Let $x^y = y^{-1}xy$. 
Suppose $j_1\neq 1$. Then 
\[
\prod_{i=1}^{g}
[x_i,y_i]=[x_{j_1},y_{j_1}]([x_1,y_1]\cdots
[x_{j_1-1},y_{j_1-1}])^{[x_{j_1},y_{j_1}]}[x_{j_1+1},y_{j_1+1}]\cdots
[x_{g},y_{g}].
\]
For each $i=1,\ldots, j_1-1$, replace the pair of generators $x_i, y_i$ with
$x_i^{[x_{j_1},y_{j_1}]},y_i^{[x_{j_1},y_{j_1}]}$. Thus we may assume that $j_1=1$. Continuing similarly, we get a new presentation of $\Pi$ of the same kind for which \eqref{eq:equal-image} holds, and hence by \cite[Lemma~6.1]{PachecoStevensonZalesskii2009} $(f,\alpha)$ is solvable.
\end{proof}

\section{Diamond $\diamond$}
In this section we prove Theorem~\ref{thm:semifreesubgroup}.
\subsection{Haran-Shapiro Induction}
\label{subsec:HS-Ind}
Let $N\leq \Pi$ be a subgroup of $\Pi$. 
Consider a FSEP
\[
\calE = (\mu_1 \colon N \to G_1, \alpha_1 \colon A\rtimes G_1 \to G_1) 
\]
for $N$. We describe a method to construct an embedding problem $\calE_{ind}$ for $\Pi$ such that a weak solution of $\calE_{ind}$ induces a weak solution of $\mathcal{E}$, and under certain conditions, a proper solution of $\calE_{ind}$ induces a proper solution of $\mathcal{E}$. 

We start by setting up the notation. Let $L\normal \Pi$ be an open normal subgroup of $\Pi$. Assume
\begin{equation}\label{ass:L-ker}
L\cap N \leq \ker \mu_1.
\end{equation} 
Let $\mu \colon \Pi \to G := \Pi/L$ be the natural epimorphism, $G_0 = NL/L \cong N/N\cap L$, and $\mu_0 = \mu|_{N}\colon N\to G_0$. Then $\mu_1$ factors as $\mu_1 = \nu \circ\mu_0$, for some canonically defined $\nu\colon G_0\to G_1$. The group $G_0$ acts on $A$ via $\nu$, i.e., $a^{g} := a^{\nu(g)}$, for all $a\in A, g\in G_0$. Thus all the maps in the following diagram are canonically defined. 

\[
\xymatrix{
	&N\ar[d]_{\mu_0}\ar@/^10pt/[dd]^{\mu_1}\\
A\rtimes G_0 \ar[r]^-{\alpha_0}\ar[d]^{\rho}
	& G_0\ar[d]_\nu \\
A\rtimes G_1\ar[r]^-{\alpha_1}
	&G_1
}
\]

The group $G$ acts on 
\[
\Ind_{G_0}^G(A) = \{ f\colon G\to A \mid f(\sigma\tau) = f(\sigma)^\tau, \ \forall \sigma\in G, \tau\in G_0\} \cong A^{(G:G_0)}
\]
by $(f^\sigma)(\sigma') = f(\sigma\sigma')$, for all $\sigma,\sigma'\in G$, $f\in \Ind_{G_0}^G(A)$. This gives rise to the so called \textbf{twisted wreath product}
\[
A\wr_{G_0} G = \Ind_{G_0}^G(A) \rtimes G.
\]
Let $\alpha\colon A\wr_{G_0}G \to G$ be the projection map. Then we have the following FSEP for $\Pi$ induced from $\calE$ (w.r.t.\ $L$ satisfying \eqref{ass:L-ker}):
\begin{equation}\label{eq:E-ind}
\calE_{ind}(L) = (\mu\colon \Pi \to G, \alpha\colon A\wr_{G_0} G \to G).
\end{equation}

\newcommand{\sh}{{\rm Sh}}
Let $\sh \colon \Ind_{G_0}^G(A) \rtimes G_0 \to A\rtimes G_0$ be defined by $\sh((f,\sigma)) = f(1)\sigma$. Clearly $\sh$ is surjective, it is also a homomorphism, since
\[
\sh(f^\sigma) = f^\sigma(1)=f(\sigma) = f(1)^\sigma = \sh(f)^\sigma.
\]
Now, a weak solution $\psi \colon \Pi \to A\wr_{G_0} G$ of $\calE_{ind}$ induces the weak solution  $\psi^{ind} = \rho \circ \sh \circ \psi|_{N}$ of $\calE$:
\[
\xymatrix@1{
N\ar[r]^-{\psi|_{N}}\ar@/_12pt/[rrr]_{\psi^{ind}}
	&\Ind_{G_0}^G(A)\rtimes G_0\ar[r]^-{\sh}
		&A\rtimes G_0\ar[r]^-\rho
			&A\rtimes G_1		
}
\]
(Note $\psi(N)\leq \Ind_{G_0}^G(A)\rtimes G_0$ since $\mu(N) = \mu_0(N) = G_0$, hence $\sh \circ\psi|_{N}$ is well defined.) 

Assume $\psi$ is surjective. In general this does not imply surjectivity of $\psi^{ind}$. The following result gives a working sufficient condition on $L$ for $\psi^{ind}$ to be surjective.

\begin{proposition}[{\cite[Proposition~4.5]{Bary-SorokerHaranHarbater}}]
\label{prop:main}%
Let $N\leq \Pi$ be profinite groups
and let
\[
\calE =( \mu_1\colon N \to G_1, \alpha_1\colon A\rtimes G_1 \to
G_1)
\]
be a FSEP for $N$. Let $D, \Pi_0, L$ be subgroups of $\Pi$ such that
\begin{labellist}\label{eq mainthm}
\item \label{eq main thma}
$D$ is an open normal subgroup of $\Pi$ with $N \cap D \le \ker\mu_1$,
\item \label{eq main thmb}
$\Pi_0$ is an open subgroup of $\Pi$ with $N\leq \Pi_0\leq ND$,
\item \label{eq main thmc}
$L$ is an open normal subgroup of $\Pi$ with $L \leq \Pi_0 \cap D$.
\end{labellist}
In particular $L\cap N\leq D\cap N\leq \ker\mu_1$, so \eqref{ass:L-ker} holds. 

Assume that there is a
closed normal subgroup $\calN$ of $\Pi$ with $\calN \leq N \cap L$ such that
there is NO nontrivial quotient $\Agag$ of $A$ through which the
action of $G_0$ on $A$ descends and for which the FSEP
\begin{equation}\label{eq-twisted-wreath-product-EP}
\bar\calE_{ind,\calN}(L) =  (\mugag\colon \Pi/\calN\to G, \alphagag\colon\Agag\wr_{G_0}G \to G),
\end{equation}
where $\mugag$ is the quotient map, $G=\Pi/L$, and $G_0 = \Pi_0/L$, is properly solvable.
Then a proper solution $\psi$ of $\calE_{ind}$ induces  a proper solution $\psi^{ind}$ of $\calE$.
\end{proposition}

\subsection{Condition ($\diamond$)}

The following result will be used in the sequel.

\begin{lemma}
\label{cor:bhh}%
\label{lem:diam-prec}%
Let $N\leq \Pi$ be profinite groups with $[\Pi:N]=\infty$ and assume there exist normal subgroups $N_1,N_2$ of $\Pi$ such that $N_1\cap N_2\leq N$, $[N_1:N_1\cap N]\geq 3$, and $[N_2:N_2\cap N]\geq 2$. Let 
\[
\calE =( \mu_1\colon N \to G_1, \alpha_1\colon A\rtimes G_1 \to
G_1)
\]
be a FSEP for $N$.  Let $L$ be an open normal subgroup of $\Pi$ satisfying 
\begin{enumerate}\renewcommand{\theenumi}{\roman{enumi}}
\item $L\cap N\leq \ker \mu_1$,
\item $[N_1 N L :NL] \geq 3$,
\item $[N_2 N L :NL] \geq 2$, and
\item $[\Pi : NL] \geq 3$.
\end{enumerate}
Let $G=\Pi/L$, $G_0 = NL/L \cong N/N\cap L$ and let
\[
\calE_{ind} = (\mu\colon \Pi \to G, \alpha\colon A\wr_{G_0} G \to G)
\]
be as defined the induced embedding problem of Equation~\eqref{eq:E-ind}. Then a proper solution $\psi$ of $\calE_{ind}$ induces a proper solution $\psi^{ind}$ of $\calE$. 
\end{lemma}

\begin{proof}
To prove the assertion we use Proposition~\ref{prop:main}.  Let $D$ be an open normal subgroup of $\Pi$ with $N\cap D\leq \ker \mu_1$, let $\Pi_0 = ND$. 
Let $L_0$ be an open normal subgroup of $\Pi$ such that for every open normal subgroup $L$ of $\Pi$ contained in $L_0$ we have 
\renewcommand{\theequation}{\arabic{equation}'}
\begin{eqnarray}
&&N_1 L, N_2 L \not\leq NL \ \mbox{(use $N_1,N_2\not\leq N$)}.\\
&&[\Pi : NL] > 2 \ \mbox{(use $[\Pi:N]> 2$)}.\\
&&(N_1 N L : NL) >  2 \ \mbox{(use $[N_1N:N]>2$)}.
\end{eqnarray}

Choose such an $L$, and let $G = \Pi/L$, $G_0 =   N/N\cap L\cong NL/L$, and $G_i = N_i/N_i\cap L \cong N_i L/L$. Then taking the  above conditions modulo $L$ gives the following conditions. 
\renewcommand{\theequation}{\arabic{equation}}\addtocounter{equation}{-3}
\begin{eqnarray}
\label{eq:not-contained}
&&G_1, G_2 \not\leq G_0.\\
&& (G:G_0) >2.
\label{eq:G-G_0>2}\\
&& (G_1G_0 :G_0)>2.
\label{eq:ind>2}
\end{eqnarray}
Let $\calN = N_1 \cap N_2 \cap L$. 

Let $\Agag$ be a non-trivial quotient of $A$ through which the action of $G_0$ descends. By Proposition~\ref{prop:main} it suffices to show that $\bar\calE_{ind,\calN}$ appearing in \eqref{eq-twisted-wreath-product-EP} is not properly solvable. 

Assume $\psi\colon \Pi \to \Agag \wr_{G_0} G$ is an epimorphism with $\alpha\circ\psi = \mu$ that factors through $F/\calN$. Then $\psi(\calN) = 1$. For $i=1,2$ put $H_i = \psi(N_i)$. Then $H_i\normal \Agag \wr_{G_0} G$ and $\alpha(H_i) = \mu (N_i) = G_i$. 
By \eqref{eq:not-contained} there is an $h_2 
\in H_2$ with $\alpha(h_2 ) \not\in 
G_0$. Recalling \eqref{eq:ind>2}, \cite[Lemma 13.7.4(a)]{FriedJarden2005} gives an $h_1\in H_1$ for which $\alpha(h_1 ) = 1$  and $[h_1 , h_2 ] \neq 1$. 

For $i=1,2$, lift $h_i$ to $y_i \in
N_i$ (i.e., $\psi(y_i ) = h_i$). Then $\mu(y_1 ) = 
\alpha(h_1 ) = 1$. So, $y_1 \in L$. Then $[y_1 , y_2 ] \in [L, N_2 ] \cap [N_1 , N_2 ] \leq L\cap  (N_1 \cap N_2 ) = 
\calN$ . So, $[h_1 , h_2 ] = [\psi(y_1 ), \psi(y_2 )] 
\in \psi(\calN ) = 1$. This contradiction proves that 
$\psi$ as above does not exist.
\end{proof}

We write $f\uparrow \infty$ for an increasing function $f\colon \bbR^+\to \bbR^+$ with $\displaystyle\lim_{x\to \infty} f(x) = \infty$.

We say that a subgroup $N$ of $\Pi$ with $[\Pi:N]=\infty$ satisfies \textbf{Condition~($\boldsymbol{\diamond}$)} in $\Pi$ if there exist normal subgroups $N_1,N_2$ of $\Pi$ such that $N_1\cap N_2 \leq N$, $[N_1:N_1\cap N]\geq 3$, $[N_2:N_2\cap N]\geq 2$, and for every $f\uparrow \infty$, $r\in \bbN$, and open subgroup $N'$ of $N$ there exists
a diagram of subgroups 
\[
\xymatrix{
N'\ar@{-}[r]
	&N\ar@{-}[r]
		&E_0\ar@{-}[r]
			&E\ar@{-}[r]
				&\Pi\\
N\cap L\ar@{-}[rr]\ar@{-}[u]
		&&L\ar@{-}[u]
}
\]
such that
\begin{enumerate}
\item $L\leq E_0\leq E$ are open in $\Pi$;
\item $L$ is normal in $E$;
\item $[N_1\cap E : N_1 \cap E_0]\geq 3$;
\item $[N_2\cap E:N_2  \cap E_0]\geq 2$;
\item $f([\Pi:E]) \geq r \cdot [E:L]$.
\end{enumerate}

In the sequel we use the notion of sparse and abundant subgroups (\cite[Defintion~2.1]{Bary-Soroker2006}) and some of their basic properties. 

\begin{definition}
A closed subgroup $M$ of a profinite group $\Pi$ of infinite index is
called \textbf{sparse} if for every $n\in\bbN$ there exists an open
subgroup $K$ of $\Pi$ containing $M$ such that for every proper open
subgroup $L$ of $K$ containing $M$ we have $[K:L]\geq n$.

It follows that one can take $K$ with arbitrarily large index in $\Pi$. See \cite[Definition~5.1]{Bary-SorokerHaranHarbater}.

A subgroup of $\Pi$ is called \textbf{abundant} if it is not open and not sparse
\end{definition}

\begin{proposition}
\label{prop:abundant-cond-s}
Let $\Pi,N,N_1,N_2$ be profinite groups such that $N,N_1,N_2$ are subgroups of $\Pi$, $N_1,N_2$ are normal in $\Pi$, $[\Pi:N]=\infty$, $N_1\cap N_2 \leq N$, $[N_1:N_1\cap N]\geq 3$, and $[N_2:N_2\cap N]\geq 2$. 
Each of the following implies that $N$ satisfies Condition ($\diamond$) in an open subgroup of $\Pi$.
\begin{labellist}
\item \label{composition_lemma_a}
$[\Pi:NN_1N_2] = \infty$.
\item \label{composition_lemma_b}
$[\Pi:NN_1N_2]<\infty$ and $NN_1$ is abundant in $\Pi$.
\item \label{composition_lemma_c}
$[\Pi:NN_1N_2]<\infty$ and $NN_2$ is abundant in $\Pi$.
\item \label{composition_lemma_d}
$[\Pi:(NN_1)\cap (NN_2)]<\infty$ and $N$ is abundant in $\Pi$.
\end{labellist}
\end{proposition}

We need two lemmas for the proof. 

\begin{lemma}\label{lem:prep}
Let $\Pi$ be a profinite group and $N$ a subgroup of $\Pi$ of infinite index. Let $N_1,N_2$ be normal subgroups of $\Pi$ such that $N_1\cap N_2 \leq N$, $[N_1:N_1\cap N]\geq 3$ and $[N_2:N_2\cap N]\geq 2$. Assume that for every  $f\uparrow \infty$, $s\in\bbN$, $\Pi$ has open subgroups $E_1
\leq E$ containing $N$ such that $f([\Pi:E])\geq s \cdot [E:E_1]!$ and
for each $i\in\{1,2\}$ either
\begin{labellist}
\item $N_i\leq E$ or \label{Preperation_Lemma_a}
\item $N_iE_1 = \Pi$ and $[E:E_1]\geq 3$.\label{Preperation_Lemma_b}
\end{labellist}
Then $N$ satisfies Condition ($\diamond$). 
\end{lemma}

\begin{proof}
Let $f\uparrow \infty$, $r\in \bbN$ and $N'$ an open subgroup of $N$. Then
there exists an open normal subgroup $D$ of $\Pi$ such that $D\cap N\leq N'$. 
Since $[N_1:N_1\cap N]\geq 3$, and $[N_2:N_2\cap N]\geq 2$, $\Pi$ has an open normal subgroup $H$ containing $N$ such that 
\begin{equation}
\label{eq:H}
[N_1:N_1\cap H]\geq 3 \qquad \mbox{and} \qquad [N_2:N_2\cap H]\geq 2.
\end{equation} 
Put $s=r\cdot[\Pi:H]![\Pi:D]$. 

Our condition
gives open subgroups $E_1\leq E$ containing $N$ such that
$f([\Pi:E])\geq s \cdot [E:E_1]!$ and for each $i\in\{1,2\}$ either
\eqref{Preperation_Lemma_a} or \eqref{Preperation_Lemma_b} holds.
Set $E_0 = H\cap E_1$. Let 
%$E_{00} = \bigcap_{\sigma\in E} E_0^\sigma$ and 
$E_{11} = \bigcap_{\sigma\in E} E_1^\sigma$ (resp., $H_{00} = \bigcap_{\sigma\in \Pi}H^\sigma$) be the normal core of $E_1$ (resp., $H$) in  $E$ (resp., $\Pi$). Finally let $L=H_{00} \cap E_{11} \cap D$. 
Then $L\leq H_{00}\cap E_{11}\leq  H \cap E_1 = E_0$. 
$$
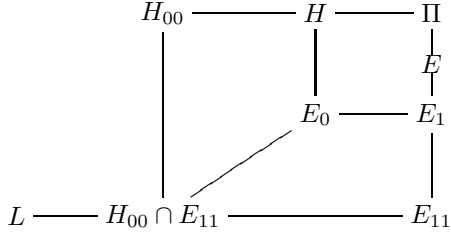
\xymatrix{%
	&H_{00}\ar@{-}[r]\ar@{-}[dd] 
		& H\ar@{-}[r]\ar@{-}[d] 
			&\Pi\ar@{-}[d]|{\mbox{$E$}} \\
		&& E_0\ar@{-}[r]\ar@{-}[dl]
			&E_1\ar@{-}[d]\\
L\ar@{-}[r]
	&H_{00}\cap E_{11}\ar@{-}[rr]&&E_{11}
 }
$$
We have
\begin{eqnarray*}
[E:L] & = & [E:H_{00}\cap E_{11} \cap D] \\
& = & [E:E_{11}][E_{11} : H_{00} \cap E_{11}][H_{00} \cap E_{11}:H_{00} \cap E_{11}\cap D]\\
&\leq& [E:E_{11}][\Pi:H_{00}][\Pi:D]\\
& \leq & [E:E_1]![\Pi:H]! [\Pi:D]\leq  \frac{1}{s} f([\Pi:E])[\Pi:H]! [\Pi:D]\\
&=&
\frac1rf([\Pi:E]).
\end{eqnarray*}

It remains to show that $[N_1\cap E:N_1\cap E_0]\geq 3$ and $[N_2\cap E:N_2\cap E_0]\geq 2$. 
First assume that $N_i\leq E$. Then, since $E_0\leq H$,
$$
[N_i\cap E: N_i\cap E_0] \geq [N_i:N_i\cap H],
$$
and we are done by \eqref{eq:H}.

Next assume that $N_i E_1 = \Pi$ and $[E:E_1]\geq 3$. Then $(N_i \cap
E) E_1 = E$, so
$$
[N_i\cap E:N_i\cap E_0]\geq [N_i\cap E:N_i\cap E_1] = [E:E_1],
$$
as needed. 
\end{proof}

\begin{lemma}\label{lem:abund-sbgrp}
Let $N$ be an abundant subgroup of a profinite group $\Pi$. Then for every $f\uparrow \infty$ and $s\in \bbN$ there exist open subgroups $N\leq E_1\leq E \leq \Pi$ such that 
$f([\Pi:E]) \geq s\cdot [E:E_1]!$ and $[E:E_1]\geq 3$. 
\end{lemma}

\begin{proof}
Since $N$ is abundant in $\Pi$, there exist $m,n\in\bbN$ such that for every open subgroup $\Pi_0$ of $\Pi$ containing $N$ with $[\Pi:\Pi_0]\geq m$ there exists
an open subgroup $\Pi_1$ of $\Pi_0$ containing $N$ such that $1<[\Pi_0:\Pi_1]\leq n$.

Let $f\uparrow \infty$ and $s\in \bbN$. By definition, $[\Pi:N]=\infty$. Thus there exists an open
subgroup $\Pi_0$ of $\Pi$ containing $N$ with $f([\Pi:\Pi_0])\geq \max\{s\cdot
n!, s\cdot 4!,f(m)\}$. In particular $f([\Pi:\Pi_0]) \geq f(m)$, thus $[\Pi:\Pi_0]\geq m$. By assumption, $\Pi_0$ has an open subgroup $\Pi_1$
containing $N$ such that $1<[\Pi_0:\Pi_1]\leq n$.

If $[\Pi_0:\Pi_1]\geq 3$, then the subgroups $E=\Pi_0$ and $E_1 = \Pi_1$ satisfy the
conclusion of the lemma. Otherwise, $[\Pi_0:\Pi_1] = 2$. By assumption $\Pi_1$
has an open subgroup $\Pi_2$ containing $N$ such that $1<[\Pi_1:\Pi_2]\leq
n$. If $[\Pi_1:\Pi_2]\geq 3$,  then the subgroups $E=\Pi_1$ and $E_1 = \Pi_2$
satisfy the conclusion of the lemma. Otherwise, $[\Pi_1:\Pi_2] = 2$, thus
$[\Pi_0:\Pi_2] = 4$, so $E=\Pi_0$, $E_1=\Pi_2$ satisfy the conclusion of the
lemma.
\end{proof}

\begin{proof}[Proof of Proposition~\ref{prop:abundant-cond-s}]
Let $f\uparrow \infty$ and $s\in \bbN$. By Lemma~\ref{lem:prep} it suffices to
find open subgroups $E_1\leq E$ of $\Pi$ containing $N$ such that
$f([\Pi:E])\geq s \cdot [E:E_1]!$ and for each $i\in\{1,2\}$ either
\begin{enumerate}\renewcommand{\theenumi}{\roman{enumi}}
\item
\label{cond:1} $N_i\leq E$ or
\item \label{cond:2}$N_i E_1
= \Pi$ and $[E:E_1]\geq 3$.
\end{enumerate}
We distinguish between the four cases:

In the first case we have $[\Pi:NN_1N_2]=\infty$. Then there exists an open subgroup
$E$ of $\Pi$ containing $NN_1N_2$ such that $f([\Pi:E])\geq s$. Put $E_1
= E$. Then $N_1,N_2\leq E$ and $[\Pi:E]\geq s\cdot [E:E_1]!$.

In the second case, we assume that $[\Pi:NN_1N_2]<\infty$ and $NN_1$ is abundant in
$\Pi$. By \cite[Corollary~2.3]{Bary-Soroker2006}, $N N_1$ is abundant in every open subgroup that contains it, so $N N_1$ is abundant in $N N_1 N_2$. Thus we can replace $\Pi$ by $NN_1N_2$ in order to assume that $\Pi = NN_1N_2$; it
suffices to proof \eqref{cond:1} and \eqref{cond:2} for this $\Pi$.
Lemma~\ref{lem:abund-sbgrp} gives open subgroups $E_1 < E$ of $\Pi$ that contain  $NN_1$ for which $f([\Pi:E])\geq s\cdot [E:E_1]!$ and
$[E:E_1]\geq 3$. Then, $N_1\leq E$ and $E_1N_2 = \Pi$.

The third case is the same as the second case, after exchanging the indices $1$ and $2$. 

In the last case we assume that $[\Pi: (NN_1)\cap (NN_2)]<\infty$ and $N$ is
abundant in $\Pi$. In particular $NN_i$ is open in $\Pi$, so 
\begin{equation}
\label{eq:infind}
[N_i:N_i\cap N]=[NN_i:N] = \infty, \qquad i=1,2.
\end{equation}
Let $\Pi' = (NN_1)\cap (NN_2)$. Then since $\Pi'$ is open in $\Pi$, it follows that $N$ is abundant in
$\Pi'$.

Put $N'_1 = N_1\cap \Pi'$ and $N'_2 = N_2\cap \Pi'$. Then $NN'_1 =
NN'_2 = \Pi'$. Since $[N_i:N'_i] < \infty$, by \eqref{eq:infind}, it follows that
$[N'_i:N'_i\cap N] = \infty$.
$$
\xymatrix{%
N_1\ar@{-}[r]&N_1N \ar@{-}[r]&\Pi\\
N'_1\ar@{-}[r]\ar@{-}[u]& \Pi'\ar@{-}[r]\ar@{-}[u]& N_2N\ar@{-}[u]\\
&N'_2\ar@{-}[r]\ar@{-}[u]&N_2\ar@{-}[u] }
$$
Replace $\Pi$ by $\Pi'$, $N_1$ by $N'_1$, and $N_2$ by $N'_2$, if
necessary, to assume that $NN_1 = \Pi$ and $NN_2 = \Pi$; it
suffices to prove  \eqref{cond:1} and \eqref{cond:2} for this $\Pi$. Lemma~\ref{lem:abund-sbgrp}
gives open subgroups $E_1\leq E$ of $\Pi$ containing $N$ with
$f([\Pi:E])\geq s\cdot [E:E_1]!$ and $[E:E_1]\geq 3$. 
Meanwhile, for $i=1,2$,  
\[
\Pi = N N_i \leq E_1 N_i \leq \Pi,
\]
hence these subgroups
satisfy \eqref{cond:2}.
\end{proof}

\subsection{Proof of Theorem~\ref{thm:semifreesubgroup}}
Let $\Pi$ be a profinite surface group of genus $g\geq 2$. We start with two lemmas.

\begin{lemma}\label{lem:semi-free-sparse}
A sparse subgroup of $\Pi$ is semi-free of countable rank.
\end{lemma}

\begin{proof}
Assume $N\leq \Pi$ is sparse.
Since $\Pi$ is finitely generated, the rank of $N$ is at most $\aleph_0$. Thus is suffices to solve any finite split embedding problem  $(\mu\colon N \to B,
\alpha \colon A\to B)$ for $N$ \cite[Lemma~3.4]{Bary-SorokerHaranHarbater}. 

Choose an open normal subgroup $D\normal \Pi$ with
$D\cap N\leq \ker \mu$ and set $H = N D$. Then $H$ is open in $\Pi$
and $\mu$ extends to an epimorphism $\mu'\colon H \to B$ by setting
$\mu'(nd)=\mu(n)$ for all $n\in N$, $d\in D$. 

Since $N$ is sparse in $\Pi$, by \cite[Lemma~2.2]{Bary-Soroker2006}, there is an open subgroup $H_0$ of $H$ that contains $N$ such that $[\Pi:H_0] \geq 2|A|^3$ and every proper open subgroup $N\leq H_1 \lneq H_0$ satisfies  
$[H_0:H_1] > |A|$. Note that $\mu_0 = \mu'|_{H_0}$ is surjective, since $\mu'(H_0) \geq \mu'(N) = B$. 

By Fact~\ref{fact:gen-grow}, we get that $H_0$ is a profinite surface group of  genus 
\[
g_0 = [\Pi:H_0] (g-1) + 1 > [\Pi:H_0] \geq 2|A|^3.
\]
By Lemma~\ref{lem:large-g-sol}, the split embedding problem 
\[
(\mu_0\colon H_0 \to B ,\alpha\colon A\to B)
\] 
is solvable; let $\gamma\colon H_0 \to A$ be a solution.
It suffices to show that $\gamma(N) = A$, or equivalently $N\ker \gamma = H_0$, since then $\gamma|_{N}$ is a solution of $(\mu,\alpha)$.
Indeed, as $[H_0 : \ker \gamma N] \leq |A|$, by the way we chose $H_0$ we have $\ker \gamma N = H_0$.
\end{proof}

\begin{lemma}
\label{lem:semi-free-abundant}
Assume $N$ satisfies Condition ($\diamond$) in $\Pi$. Then $N$ is semi-free. 
\end{lemma}

\begin{proof}
Since $\Pi$ is finitely generated  we get that $N$ is countably generated. Hence it suffices to show that every finite split embedding problem
\[
\calE = (\mu_1 \colon N\to G_1, \alpha_1 \colon A\rtimes G_1\to G_1)
\]
with $A\neq 1$ is solvable. 

Let $f(x)=log x$ and take $N' = \ker \mu_1$. Choose $r$ such that 
\begin{equation}
\label{eq:inequality}
e^{r y} \geq 2|A|^{3y} y^3,\  \forall y\geq 2.
\end{equation}
By Condition~($\diamond$), there exist normal subgroups $N_1, N_2$ of $\Pi$ such that $N_1\cap N_2 \leq N$, $[N_1:N_1\cap N]\geq 3$, $[N_2:N_2\cap N]\geq 2$ and a diagram of subgroups 
\[
\xymatrix{
\ker \alpha_1\ar@{-}[r]
	&N\ar@{-}[r]
		&E_0\ar@{-}[r]
			&E\ar@{-}[r]
				&\Pi\\
N\cap L\ar@{-}[rr]\ar@{-}[u]
		&&L\ar@{-}[u]
}
\]
satisfying the following conditions:
\begin{enumerate}
\item $L\leq E_0\leq E$ are open in $\Pi$.
\item $L$ is normal in $E$.
\item $[N_1\cap E : N_1 \cap E_0]\geq 3$.
\item $[N_2\cap E:N_2  \cap E_0]\geq 2$.
\item $\log([\Pi:E]) \geq r \cdot [E:L]$, or equivalently, $[\Pi:E] \geq e^{r\cdot [E:L]}$.
\end{enumerate}
Then $E$ is a surface group of genus
\begin{equation}
\label{eq:large-ind}
g_0 =[ \Pi:E] (g-1) +1 \geq e^{r\cdot [E:L]}\geq 2|A|^{3[E:L]} [E:L]^3.
\end{equation}

Let $N_i' = N_i\cap E$. 
We apply Lemma~\ref{lem:diam-prec} with $E$ replacing $\Pi$ and $N_i'$ replacing $N_i$. Let 
$G = E/L$, $G_0 = NL/L\cong N/N\cap L$, and 
\[
\calE_{ind} = (\mu \colon E \to G, \alpha \colon A\wr_{G_0} G \to G)
\]
the induced embedding problem. We claim that the conditions of the lemma are satisfied. Indeed, $L\cap N\leq \ker \mu_1$ by the diagram. By \eqref{prf:diam-large-ind} and since $L\leq E_0$ 
we have
\[
[N_i'NL:NL] = [N_i' : N_i'\cap NL] = [N_i\cap E : (N_i\cap E)\cap NL]\geq  [N_i\cap E : N_i\cap E_0].
\] 
Finally 
\[
[E:NL] \geq [E:E_0] \geq [ N_1\cap E :  N_1\cap E_0]\geq 3.
\]
Now by \eqref{eq:large-ind} we have that $g_0 \geq 2|A\wr_{G_0} G|^3$, hence by  Lemma~\ref{lem:large-g-sol}, $\calE_{ind}$ is properly solvable, and thus by Lemma~\ref{lem:diam-prec} so is  $\calE$.
\end{proof}

\begin{proof}[Proof of Theorem~\ref{thm:semifreesubgroup}]
First we may assume that $[N_1 N:N] = \infty$. Indeed, if $[N_1 N: N] < \infty$, then $\Pi$ has an open subgroup $N_2'$ such that $N_2'\cap (N_1N)\leq N$. Then $N_1\cap N_2' \leq N$ and $[N_2'N:N]=\infty$. Replace $N_1$ with $N_2'$ and $N_2$ with $N_1$ to get the assumption.

Note that by Lemma~\ref{lem:semi-free-abundant}, if one of the conditions of Proposition~\ref{prop:abundant-cond-s} is satisfied, then $N$ is semi-free. Hence we assume that none of them holds.  

If $[\Pi : (NN_1) \cap (NN_2)] < \infty$, then the negation of Condition~\ref{composition_lemma_d} of Proposition~\ref{prop:abundant-cond-s} gives that $N$ is sparse in $\Pi$. Hence $N$ is free of countable rank (Lemma~\ref{lem:semi-free-sparse}). 
Assume that $[\Pi : (NN_1) \cap (NN_2)] = \infty$. W.l.o.g. $[\Pi:NN_1]=\infty$. Then the negation of (\eqref{composition_lemma_a} $\vee$ \eqref{composition_lemma_b}) gives that $NN_1$ is sparse in $\Pi$. Then Lemma~\ref{lem:semi-free-sparse} gives that $NN_1$ is free of countable rank. 

Put $N'_2 = (NN_1)\cap N_2$. Then $N_1, N'_2\normal NN_1$, $N_1\cap N'_2\leq
N$ and $N_1\not \leq N$. If $N'_2\not\leq N$, then the diamond theorem for free groups (\cite[Theorem~25.4.3]{FriedJarden2005}) gives that 
$N$ is free of countable rank. 

We are left with the case  $N'_2\leq N$. Then
$N=NN_1\cap NN_2$. By the negation of \eqref{composition_lemma_a} of Proposition~\ref{prop:abundant-cond-s}, $[\Pi:NN_1N_2] < \infty$, hence
\[
[\Pi:NN_2]=[\Pi:NN_1N_2][NN_1:N]=[N_1:N_1\cap N]=\infty.
\]
\[
\xymatrix{
N_2\ar@{-}[r]
	&NN_2\ar@{-}[r]^{\infty}
		&NN_1N_2\ar@{-}[r]^-{<\infty}
			&\Pi\\
N_2'\ar@{-}[u]\ar@{-}[r]
	&N\ar@{-}[u]\ar@{-}[r]^{\infty}
		&NN_1\ar@{-}[u]
}
\]
The negation of \eqref{composition_lemma_c} of Proposition~\ref{prop:abundant-cond-s} gives that $NN_2$ is sparse in $\Pi$, hence in the open subgroup$NN_1N_2$ of $\Pi$ (\cite[Corollary~2.3]{Bary-Soroker2006}). 
But since $NN_1/N_2' \cong NN_1N_2/N_2$, this implies that $N$ is sparse in the free group $NN_1$, and hence $N$ is free of countable rank (\cite[Lemma~2.4]{Bary-Soroker2006}).
\end{proof}

\section{Applications}
\subsection{Proof of Theorem~\ref{thm:sub-semi}}
Let $\Pi$ be a surface group of genus $g\geq 2$ and let $N\lhd \Pi$ be a normal subgroup of infinite index such that $\Pi/N$ is not hereditary just infinite . We need to prove that $N$ is contained in a semi-free normal subgroup. 

If there exists a normal subgroup $N\lneqq M\normal \Pi$ with $[\Pi:M]=\infty$, then there exists $N\leq U\lhd \Pi$ open in $\Pi$ such that $M\cap U \neq M$ (recall that $N$ is the intersection of all open subgroups containing it). 
So $M\cap U$ is semi-free by Theorem~\ref{thm:semifreesubgroup}, and we are done.

Therefore we can assume that $J = \Pi/N$ is just infinite. By \cite[Theorem~3(b)]{Grigorchuk2000}, there exists an open normal subgroup $J_0$ of $J$ such that either $J_0$ is hereditarily just infinite, which is not possible by assumption, or $J_0 = K_1\times K_2$, where $K_i$ is infinite group, $i=1,2$.

Let $\Pi_0, N_1,N_2$ be the respective preimages of $J_0,K_1.K_2$ under the map $\Pi\to J$. Then $\Pi_0$ is a surface group of genus $\geq 2$ and $N = K_1\cap K_2$. So by Theorem~\ref{thm:semifreesubgroup}, $N$ is semi-free.
\qed

\begin{remark}
Let $N$ be a normal subgroup of $\Pi$ such that $\Pi/N$ is hereditarily just infinite. We do not know whether $N$ is necessarily semi-free.
\end{remark}

\subsection{Proof of Corollary~\ref{cor:LubvdDries}}
 Let $\Pi$ be a surface group of genus at least $2$, $M$ a normal subgroup of $\Pi$ of infinite index, and $N$ a proper open subgroup of $M$. There exists an open normal subgroup $U\lhd \Pi$ such that $U\cap M\leq N$, so by the Theorem~\ref{thm:semifreesubgroup}, $N$ is semi-free. 
\qed

\subsection{Some examples} 
\label{sec:examples}
Let $\Pi$ be a surface group of genus at least $2$ and $N$ a closed subgroup of infinite index. 
The following result provides many interesting examples of semi-free subgroups of a surface group.

\begin{proposition}
If $N\lhd \Pi$ and every open subgroup of $\Pi/N$ is generated by $d$ elements, for some $d\geq 1$, or if $N$ is sparse in $\Pi$, then $N$ is semi-free. 
\end{proposition} 

\begin{proof}
Let 
\[
\calE(N) = (\mu_0\colon N\to A, \alpha_0\colon C\rtimes A \to A) 
\]
be a FSEP for $N$. Assume first that every open subgroup of $\Pi/N$ is generated by $d$ elements. Let $r,n\geq 1$ be given such that $n> ((|C||A|)!)^d$ and $r\geq 2|A|^3|C|^{3n}$.

Let $L\lhd \Pi$ be an open subgroup of $\Pi$ such that $L\cap N \leq \ker\mu_0$. Let $\tilde \Pi$ be an open subgroup of $\Pi$ such that $N\leq \tilde \Pi \leq L N$ and such that $[\Pi:\tilde\Pi] \geq r$. Then we can extend $\mu_0$ to $\mu\colon \tilde \Pi\to A$ by $\mu(nl) = \mu_0(n)$, for every $n\in N$, $l\in L$, for which $nl\in \tilde \Pi$. By Fact~\ref{fact:gen-grow} the genus of $\tilde \Pi$ is at least $r$. Without loss of generality we can replace $\Pi$ with $\tilde\Pi$ to assume $\mu$ is defined and $g\geq r$.  (Note that the rank of  $\tilde\Pi/N$ is bounded by the rank of $\Pi/N$.)

Consider the FSEP 
\[
\calE_n(\Pi) = (\mu\colon \Pi\to A, \alpha\colon C^n\rtimes A \to A),
\]
where $A$ acts component-wise on $C^n$. Since $g\geq r \geq 2|A|^3|C|^{3n}$, by Lemma~\ref{lem:large-g-sol}, there exists a proper solution $\Psi\colon \Pi \to C^n\rtimes A$ of $\calE_n(\Pi)$. For each $i=1,\ldots, n$, let $\psi_i$ be the composition of $\Psi$ with the projection $C^n\rtimes A\to C\rtimes A$ on the $i$th coordinate. Let $L_i = \ker \psi_i$. Then $L_iL_j=\ker\mu$, for every $i\neq j$.

If $L_i  N = \Pi$ for some $i$, then $\psi_i(N) = C\rtimes A$, so $\psi_i|_{N}$ is a proper solution of $\mathcal{E}(N)$, and we are done. 

Otherwise, assume that $L_i  N \neq \Pi$ for every $i$. But since $(L_i N) (L_j N) = (L_i L_j) N = \ker\mu N = \Pi$, we get that $L_i N$ are distinct subgroups of index $\leq |C||A|$. So $\Pi/N$ has at least $n> ((|C||A|)!)^d$ open subgroups of index  $\leq |C||A|$. This is a contradiction because each such a subgroup induces a distinct homomorphism to the symmetric group $S_{|C||A|}$ defined by the action on the cosets, and the number of these homomorphisms is bounded by $ ((|C||A|)!)^d$.

Next assume that $N$ is sparse in $\Pi$. Replace $\Pi$ by an open subgroup $\tilde \Pi$ of index $[\Pi:\tilde \Pi]\geq 2 |C|^3|A|^3$ that  contains $N$ such that $\tilde \Pi$ has no proper subgroups of index $\leq |C||A|$ that contain $N$. Then arguing as above with $n=1$, we get that $L_1 N \leq \Pi$ and $[\Pi:L_1N]\leq |C||A|$, so $L_1N = \Pi$. So $\psi_1|_{N}$ is a proper  solution of $\mathcal{E}(N)$. 
\end{proof}

\begin{examples}
Each of the following conditions implies that $N$ is semi-free. 
\begin{enumerate}
\item $\Pi/N = \mathbb{Z}_p$ (every subgroup is cyclic)
\item $\Pi/N = K_1 \times K_2$ ($N$ is the intersection of the preimages of $K_1,K_2$ in $\Pi$, hence by Theorem~\ref{thm:semifreesubgroup}, is semi-free). 
\item $\Pi/N$ is abelian ($\Pi/N$ is either $\mathbb{Z}_p$ or direct product). 
\item $\Pi/N$ is pro-nilpotent but not pro-$p$ ($\Pi/N$ is a direct product). 
\item $[\Pi:N] = \prod_{p} p^{n(p)}$, where $0\leq n(p)<\infty$ (this implies that $N$ is sparse in $\Pi$). 
\end{enumerate}
\end{examples}

Notice that (2) gives a new proof that the congruence kernel of an arithmetic lattice in $SL_2(\mathbb{R})$ is a free profinite group of countable rank, see \cite{Zalesskii2005} for more details.

\bibliographystyle{amsplain}
\providecommand{\bysame}{\leavevmode\hbox to3em{\hrulefill}\thinspace}
\providecommand{\MR}{\relax\ifhmode\unskip\space\fi MR }
\providecommand{\MRhref}[2]{%
  \href{http://www.ams.org/mathscinet-getitem?mr=#1}{#2}
}
\providecommand{\href}[2]{#2}

\end{document}